\documentclass[reqno]{amsart}

\usepackage{amsmath}
\usepackage{amssymb}
\usepackage{amsthm}
\usepackage{mathrsfs}
\usepackage{mathtools}
\usepackage[hmarginratio=1:1]{geometry}
\usepackage{graphicx}
\usepackage{longtable}
\usepackage{booktabs}
\usepackage{float}
\usepackage{aliascnt}
\usepackage[font=footnotesize]{caption}
\usepackage{enumerate}
\usepackage{cleveref}
\usepackage{enumitem}

\newcommand{\ZZ}{\mathbb{Z}}
\newcommand{\QQ}{\mathbb{Q}}
\newcommand{\CC}{\mathbb{C}}
\newcommand{\oo}{\mathfrak{o}}
\newcommand{\pp}{\mathfrak{p}}
\newcommand{\dd}{\mathfrak{d}}

\newcommand{\trace}{\text{Tr}}
\newcommand{\norm}{\text{N}}
\newcommand{\legendre}[2]{\left(\frac{#1}{#2}\right)}
\renewcommand{\pmod}[1]{\text{ (mod }#1)}
\newcommand{\ord}{\text{ord}}

\theoremstyle{definition}
\newtheorem{theorem}{Theorem}[section]

\newtheorem{lemma}[theorem]{Lemma}

\newtheorem{remark}[theorem]{Remark}

\numberwithin{equation}{section}
\numberwithin{table}{section}

\allowdisplaybreaks[4]

\title[Local densities of quadratic polynomials]{Local densities of quadratic polynomials over non-dyadic fields and unramified dyadic fields}

\author{Zilong He}
\address{School of Computer Science and Technology, Dongguan University of Technology, Dongguan 523808, P.R. China}
\email{zilonghe@connect.hku.hk}

\author{Zichen Yang}
\address{Department of Mathematics, University of Hong Kong, Pokfulam, Hong Kong}
\email{zichenyang.math@gmail.com}

\thanks{The first author was supported by the National Natural Science Foundation of China (Grant No. 12301013) and
the start-up foundation of Dongguan University of Technology (Grant No. 221110236). The research of the second author was supported by the Dissertation Year Fellowship of the University of Hong Kong.}

\date{\today}

\begin{document}

\begin{abstract}
In this paper, we prove formulas for local densities of quadratic polynomials over non-dyadic fields and over unramified dyadic fields.
\end{abstract}

\maketitle

\section{Introduction}

The representation problem is one of the central problems in the arithmetic theory of quadratic forms. Specifically, fix the ring of integers $\oo$ of a number field $K/\QQ$ of degree $d$ and discriminant $D_K$. For a quadratic form $Q\in\oo[x_1,\ldots,x_r]$ and an element $n\in\oo$, we hope to know if there exists a solution $(x_1,\ldots,x_r)\in\oo^r$ to the Diophantine equation $Q(x_1,\ldots,x_r)=n$. Such a solution is called a representation of $n$ by $Q$. More ideally, one may hope to have an exact formula or at least an asymptotic formula for the number of representations $r_Q(n)$ when it is finite, for example, when $K$ is totally real and $Q$ is totally positive. By Siegel's fundamental works \cite{siegel1935uber,siegel1936uber,siegel1937uber,siegel1951indefinite,siegel1951indefinite2} and the generalization by Weil \cite{weil1965formule}, in order to obtain such a formula for $r_Q(n)$, one way is to consider a weighted average $a_E(n)$ of the number of representations $r_{Q_1}(n),\ldots,r_{Q_g}(n)$, where the quadratic forms $Q_1,\ldots,Q_g$ form a complete set of representatives of the classes in the genus of $Q$. Precisely, we have 
$$
r_Q(n)=a_E(n)+a_G(n),
$$
with
$$
a_E(n)\coloneqq\frac{1}{m(Q)}\sum_{1\leq i\leq g}\frac{r_{Q_i}(n)}{|O(Q_i)|},
$$
where we denote by $m(Q)$ the mass of $Q$ and denote by $O(Q)$ the orthogonal group of $Q$, and $a_G(n)$ is the $n$-th Fourier coefficient of a fixed cusp form $G$ depending only on $Q$. 

To understand the behaviour of $a_E(n)$, Siegel proved a product formula (c.f. \cite[(5.3.2)]{shimura1999number}):
\begin{equation}
\label{eqn::siegelminkowskismith}
a_E(n)=\frac{c_r(2\pi)^{\frac{rd}{2}}N_{K/\QQ}(n)^{\frac{r}{2}-1}}{D_K^{\frac{r-1}{2}}\Gamma(r/2)^dN_{K/\QQ}(\det(Q))^{\frac{1}{2}}}\prod_{\pp\neq\infty}\beta_{\pp}(n;Q\otimes_{\oo}\oo_{\pp}),
\end{equation}
with $c_r=1$ for $r>2$ and $c_r=1/2$ for $r=2$, where $\pp$ runs over all prime ideals of $K$, and $\beta_{\pp}(n;Q)$ is the local density of a quadratic form $Q$ over $K_{\pp}$, defined in (\ref{eqn::localdensity}). 

This product formula (\ref{eqn::siegelminkowskismith}) reduces the representation problem to the problem of calculating local densities $\beta_{\pp}(n;Q)$ for an integral quadratic form $Q$ over $K_{\pp}$, which has been studied extensively by previous researchers. Siegel computed the local density for quadratic forms associated with unimodular lattices over $\QQ_p$ with $p\neq2$. In \cite{yang1998explicit}, Yang further gave an explicit formula for local densities of any quadratic forms over $\QQ_p$ for any prime number $p$. More generally, if one is interested in local densities of higher dimensional representations of quadratic forms, one may refer to \cite{katsurada1995certain,katsurada1999explicit,katsurada1997recursion,kitaoka1983note}, and in the case when $K_{\pp}=\QQ_p$ with $p\neq2$, Hironaka and Sato \cite{sato2000local} gave an explicit formula in full generality. See Kitaoka's book \cite[Section 5.6]{kitaoka1999arithmetic} for a systematic treatment. The results are less complete when $K_{\pp}\neq\QQ_p$. Shimura \cite{shimura1999number} proved an explicit formula for quadratic forms associated with maximal lattices. In \cite{hanke2004local}, Hanke gave reduction maps for computing local densities, which can be implemented as an algorithm but gives no explicit formula. Recently, Thompson \cite{thompson2024sum} used his method to calculate the local densities for sum of four squares over quadratic fields.

More generally, one may consider the representation problem of a quadratic polynomial. In this case, an extension of Siegel's product formula was given in \cite[(1.15)]{shimura2004inhomogeneous}. When $K=\QQ$, the second author with Kane \cite{kane2023finiteness} extended the explicit formula in \cite{yang1998explicit} to the case of quadratic polynomials. Because computing local densities of quadratic polynomials plays an important role in arithmetic applications \cite{kane2020universal,kane2023finiteness}, it is also reasonable to hope for a formula for local densities of quadratic polynomials.

The goal of this paper is to give explicit formulas for local densities of integral quadratic polynomials over number fields. When $K_{\pp}$ is non-dyadic, by modifying Yang's approach in \cite{yang1998explicit}, we generalize his formula \cite[Theorem 3.1]{yang1998explicit} to quadratic polynomials over $K_{\pp}$ in Theorem \ref{thm::localdensityp}. When $K_{\pp}$ is dyadic, the calculation becomes technically complicated due to the difficulty in the evaluation of Gauss integrals. When $K_{\pp}$ is unramified, using Gurak's work \cite[Section 5]{gurak2010gauss}, we compute the Gauss integrals and obtain the formulas for quadratic polynomials over unramified dyadic fields in Theorem \ref{thm::localdensity2}.

The rest of the paper is organized as follows. In Section \ref{sec::pre}, we define the local density in terms of an integral over local fields and reduce the problem using Jordan canonical form theorem. Then, we prove formulas for local densities of quadratic polynomials over non-dyadic fields in Section \ref{sec::nondyadic}, and over unramified dyadic fields in Section \ref{sec::unramifieddyadic}.

\section{Preliminaries}
\label{sec::pre}

Since the discussion from now on is purely local, by abuse of notation, we denote by $K$ a finite extension over the $p$-adic field $\QQ_p$ of ramification index $e$ and inertial degree $f$, where $p$ is a rational prime number. We denote by $\oo$, $\pp$, $\kappa$ the ring of integers, the unique prime ideal, and the residue class field of $K$, respectively. The trace map and the norm map from $K$ to $\QQ_p$ is denoted by $\trace=\trace_{K/\QQ_p}$ and $\norm=\norm_{K/\QQ_p}$, respectively. Let $\pp^{\dd}$ be the different of the extension $K/\QQ_p$ for an integer $\dd\in\ZZ$, which is the inverse of the codifferent of the extension $K/\QQ_p$ given by
\begin{equation}
\label{eqn::codifferent}
\pp^{-\dd}=\left\{x\in K~|~\trace_{K/\QQ_p}(x\oo)\subseteq\ZZ_p\right\}.
\end{equation}
For a uniformizer $\pi$ of $K$, we define an exponential function $e_{\pi}\colon K\to\CC$ by
$$
e_{\pi}(\alpha)\coloneqq e_p\left(\trace_{K/\QQ_p}(\alpha/\pi^{\dd})\right),
$$
for any element $\alpha\in K$, where for any element $\alpha\in\QQ_p$, we set $e_p(\alpha)\coloneqq e^{-2\pi ia}$ with $a\in\bigcup_{t\geq1}p^{-t}\ZZ$ such that $\alpha-a\in\ZZ_p$. The function $e_{\pi}$ has the following properties:
\begin{enumerate}
\item For any elements $\alpha_1,\alpha_2\in K$, we have $e_{\pi}(\alpha_1+\alpha_2)=e_{\pi}(\alpha_1)e_{\pi}(\alpha_2)$,
\item If $\alpha\in\oo$ then $e_{\pi}(\alpha)=1$,
\item If $\alpha\not\in\oo$, then there exists an element $\beta\in\oo$ such that $e_{\pi}(\alpha\beta)\neq1$,
\end{enumerate}
which will be used frequently throughout this paper.

\subsection{Local densities of quadratic polynomials}

Fix an integer $r\geq1$. Put $\mathbf{x}\coloneqq(x_1,\ldots,x_r)^{\mathrm{t}}$. Any quadratic polynomial $Q(\mathbf{x})$ in $r$ variables over $K$ is of the form
\begin{equation}
\label{eqn::quadraticpolynomial}
Q(\mathbf{x})=\mathbf{x}^{\mathrm{t}}A\mathbf{x}+\mathbf{b}^{\mathrm{t}}\mathbf{x}+c,
\end{equation}
where $A\in K^{r\times r}$ is non-singular and symmetric, $\mathbf{b}\in K^r$, and $c\in K$. In particular, if $Q(\mathbf{x})$ is a quadratic form, then $\mathbf{b}=\mathbf{0}$ and $c=0$. We adopt the non-classical notion of integrality, that is, a quadratic polynomial $Q$ is called integral if $Q(\mathbf{x})\in\oo[\mathbf{x}]$.

Since $K$ is locally compact, there exists a unique Haar measure $\mathrm{d}x$ on $K$, which is normalized so that the area of $\oo$ is $1$, that is, $\int_{\oo}\mathrm{d}x=1$. Since a Haar measure is invariant under translation, we have $\mathrm{d}(\alpha+x)=\mathrm{d}x$ and $\mathrm{d}(\alpha x)=\norm(\pp)^{-\ord_{\pp}(\alpha)}\mathrm{d}x$ for $\alpha\in K^{\times}$. Therefore we have $\int_{\pp^t}\mathrm{d}x=\norm(\pp)^{-t}$ for $t\in\ZZ$. For $n\in\oo$, the local density of a quadratic polynomial $Q$ at $\pp$ is defined by
\begin{equation}
\label{eqn::localdensity}
\beta_{\pp}(n;Q)\coloneqq\int_{K}\int_{\oo^r}e_{\pi}\left(\sigma(Q(x_1,\ldots,x_r)-n)\right)\mathrm{d}x_1\cdots\mathrm{d}x_r\mathrm{d}\sigma,
\end{equation}
with $\mathrm{d}x_1=\cdots=\mathrm{d}x_r=\mathrm{d}\sigma=\mathrm{d}x$. Note that the integral is independent of the choice of the uniformizer $\pi$. 

Throughout this paper, we denote by $\chi_S$ the characteristic function of a subset $S\subseteq K$, that is, $\chi_S(\sigma)=1$ if $\sigma\in S$ and $\chi_S(\sigma)=0$ otherwise.

\begin{lemma}
\label{thm::padiccirclemethod}
For $\sigma\in K$ and $t\in\ZZ$, we have
$$
\int_{\pp^t}e_{\pi}(\sigma x)\mathrm{d}x=\frac{\chi_{\pp^{-t}}(\sigma)}{\norm(\pp)^t}.
$$
\end{lemma}
\begin{proof}
Set
$$
I(\sigma)\coloneqq\int_{\oo}e_{\pi}(\sigma x)\mathrm{d}x.
$$
Replacing $x$ by $\pi^tx$, we have
$$
\int_{\pp^{t}}e_{\pi}(\sigma x)\mathrm{d}x=\frac{1}{\norm(\pp)^t}\int_{\oo}e_{\pi}(\sigma \pi^tx)\mathrm{d}x=\frac{I(\sigma\pi^t)}{\norm(\pp)^{t}}.
$$
So it suffices to show that $I(\sigma)=\chi_{\oo}(\sigma)$. If $\sigma\in\oo$, then $I(\sigma)=1$. If $\sigma\not\in\oo$, then $\sigma/\pi^{\dd}\not\in\pp^{-\dd}$. It follows that there exists some $y\in\oo$ such that $e_{\pi}(\sigma y)\neq1$. Replacing $x$ by $x+y$, we see that
$$
I(\sigma)=e_{\pi}(\sigma y)I(\sigma),
$$
which implies $I(\sigma)=0$, as desired.
\end{proof}

\begin{lemma}
\label{thm::sumcirclemethod}
For $\sigma\in\oo$, we have
$$
\sum_{x\in\oo/\pp}e_{\pi}\left(\frac{\sigma x}{\pi}\right)=\norm(\pp)\chi_{\pp}(\sigma).
$$
\end{lemma}
\begin{proof}
By Lemma \ref{thm::padiccirclemethod}, we have
$$
\int_{\pp^{-1}}e_{\pi}(\sigma x)\mathrm{d}x=\sum_{a\in\oo/\pp}\int_{a/\pi+\oo}e_{\pi}(\sigma x)\mathrm{d}x=\sum_{a\in\oo/\pp}e_{\pi}\left(\frac{\sigma a}{\pi}\right)\int_{\oo}e_{\pi}(\sigma x)\mathrm{d}x=\sum_{a\in\oo/\pp}e_{\pi}\left(\frac{\sigma a}{\pi}\right).
$$
Therefore, applying Lemma \ref{thm::padiccirclemethod} again with $t=-1$, we obtain the desired result.
\end{proof}

\subsection{Jordan canonical forms}

By definition, the local densities are invariant if we apply an invertible linear transformation in $\text{GL}_r(\oo)$ to the quadratic polynomial. Therefore we may assume that the quadratic polynomial is given in a reduced form. This is achieved in the following lemma.

\begin{lemma}
\label{thm::jordan}
Let $Q$ be an integral quadratic form over $K$. The following results hold:
\begin{enumerate}
\item If $K$ is non-dyadic, then there exists an invertible linear transformation $T\in\text{GL}_r(\oo)$ such that $Q\circ T$ is an integral diagonal form.
\item If $K$ is dyadic and unramified, then there exists an invertible linear transformation $T\in\text{GL}_r(\oo)$ such that $Q\circ T$ is a sum of quadratic form $\pi^{t_i}Q_i$ with $t_i\geq0$, where $Q_i$ is of the form $ux^2$, $uxy$, or $u(x^2+xy+\rho y^2)$ with $u\in\oo^{\times}$, where $\rho\in\oo^{\times}$ is a fixed unit depending on the field $K$. 
\end{enumerate}
\end{lemma}
\begin{proof}
The assertion (1) was proved \cite[Lemma 2.1]{hanke2004local}. To prove the assertion (2), since $K$ is unramified, by \cite[Lemma 2.1]{hanke2004local}, there exists an invertible linear transformation $T\in\text{GL}_r(\oo)$ such that $Q\circ T$ is a sum of forms $\pi^{t_i}Q_i$ with $t_i\geq0$, where $Q_i$ is either of the form $ux^2$ with $u\in\oo^{\times}$ or of the form $u(ax^2+xy+by^2)$ with $a,b\in\oo$ and $u\in\oo^{\times}$ such that $\ord_{\pp}(a)=\ord_{\pp}(b)$. Using \cite[93:11]{o2013introduction}, since $2ax^2+2xy+2by^2$ is an even unimodular quadratic form, by composing an invertible linear transformation in $\text{GL}_r(\oo)$, we can convert it into either $2xy$ or $2x^2+2xy+2\rho y^2$ with a unit $\rho\in\oo^{\times}$. This proves the assertion (2).
\end{proof}
\begin{remark}
The unit $\rho$ is chosen so that the quadratic defect of $1+4\rho$ is $4\oo$. See \cite[\S 93]{o2013introduction} for details. By \cite[63:3]{o2013introduction}, we can take $\rho=1$ when $f$ is odd. When $f=2$, we can take $\rho=\zeta$, where $\zeta$ is a primitive cubic root of unity.
\end{remark}

\section{Non-dyadic cases}
\label{sec::nondyadic}

In this section, we give an explicit formula for the local densities when $K$ is non-dyadic. First, we have to study Gauss sums defined by
$$
G_{\pi}(\sigma)\coloneqq\sum_{x\in\kappa}e_{\pi}\left(\frac{\sigma x^2}{\pi}\right),
$$
with $\sigma\in\oo^{\times}$. For $x\in\oo$, we define the generalized Legendre symbol by
$$
\legendre{x}{\pp}\coloneqq
\begin{dcases}
1,&\text{ if }x\in\oo^{\times},x\equiv y^2\pmod{\pp}\text{ for some }y\in\oo^{\times},\\
-1,&\text{ if }x\in\oo^{\times},x\not\equiv y^2\pmod{\pp}\text{ for any }y\in\oo^{\times},\\
0,&\text{ if }x\in\pp.
\end{dcases}
$$

The following properties of Gauss sums are standard.

\begin{lemma}
\label{thm::gausssumproperty}
For $\sigma\in\oo^{\times}$, we have
\begin{equation}
\label{eqn::firstformula}
G_{\pi}(\sigma)=\sum_{x\in\kappa^{\times}}\legendre{x}{\pp}e_{\pi}\left(\frac{\sigma x}{\pi}\right),  
\end{equation}
and
\begin{equation}
\label{eqn::secondformula}
G_{\pi}(\sigma)=\epsilon_{\pi}^3\legendre{\sigma}{\pp}\norm(\pp)^{\frac{1}{2}},
\end{equation}
with $\epsilon_{\pi}\in\{\pm1,\pm i\}$.
\end{lemma}
\begin{proof}
By Lemma \ref{thm::sumcirclemethod}, we have 
\begin{equation}
\label{eqn::orthogonality}
\sum_{x\in\kappa}e_{\pi}\left(\frac{\sigma x}{\pi}\right)=0.
\end{equation}
Since $\kappa^{\times}$ is cyclic of an even order, by (\ref{eqn::orthogonality}), we see that
$$
G_{\pi}(\sigma)=\sum_{x\in\kappa}\left(1+\legendre{x}{\pp}\right)e_{\pi}\left(\frac{\sigma x}{\pi}\right)=\sum_{x\in\kappa}\legendre{x}{\pp}e_{\pi}\left(\frac{\sigma x}{\pi}\right)=\sum_{x\in\kappa^{\times}}\legendre{x}{\pp}e_{\pi}\left(\frac{\sigma x}{\pi}\right).
$$
This finishes the proof of (\ref{eqn::firstformula}). It follows that
$$
G_{\pi}(\sigma)=\legendre{\sigma}{\pp}G_{\pi}(1).
$$
So to prove (\ref{eqn::secondformula}), it suffices to calculate $G_{\pi}(1)^2$. Using (\ref{eqn::orthogonality}) again, we see that
$$
G_{\pi}(1)^2=\sum_{x,y\in\kappa}\legendre{xy}{\pp}e_{\pi}\left(\frac{x+y}{\pi}\right)=\sum_{y\in\kappa}\legendre{y}{\pp}\sum_{x\in\kappa}e_{\pi}\left(\frac{x(1+y)}{\pi}\right)=\legendre{-1}{\pp}\norm(\pp).
$$
Thus (\ref{eqn::secondformula}) follows from taking square roots of both sides.
\end{proof}
\begin{remark}
One may replace $\epsilon_{\pi}^3$ by $\epsilon_{\pi}$ with a different choice of $\epsilon_{\pi}$. Our choice guarantees that when $K=\QQ_p$ and $\pi=p$, we have
$$
\epsilon_p=
\begin{dcases}
1,&\text{ if }p\equiv1\pmod{4},\\
i,&\text{ if }p\equiv3\pmod{4},
\end{dcases}
$$
which is aligned with the formula in \cite[(1.10)]{yang1998explicit}. In general, the value of the fourth root of unity $\epsilon_{\pi}$ can be determined by \cite[Theorem 6.1]{szechtman2002quadratic}.
\end{remark}

By Lemma \ref{thm::jordan}, we see that any integral quadratic form can be converted to a diagonal form by applying an invertible linear transformation. Applying this theorem to a quadratic polynomial gives a diagonal form with linear terms and a constant term. We may further assume that the constant term is zero because $\beta_{\pp}(n;Q)=\beta_{\pp}(n-c;Q-c)$ for $c\in\oo$. Overall, we can suppose that the quadratic polynomial is given by
$$
Q(\mathbf{x})=\sum_{1\leq i\leq r}(b_ix_i^2+c_ix_i),
$$
with $b_1,c_1,\ldots,b_r,c_r\in\oo$. It follows that we can rewrite the local density of $Q$ by
$$
\beta_{\pp}(n;Q)=\int_{K}e_{\pi}\left(-\sigma n\right)\prod_{i=1}^rI_{\pi}(\sigma b_i,\sigma c_i)\mathrm{d}\sigma,
$$
where we set
\begin{equation}
\label{eqn::gaussintegral1}
I_{\pi}(\sigma,\tau)\coloneqq\int_{\oo}e_{\pi}\left(\sigma x^2+\tau x\right)\mathrm{d}x.
\end{equation}

First we give a formula for $I_{\pi}(\sigma,0)$.

\begin{lemma}
\label{thm::gaussintegral1p}
For $\sigma=u\pi^t\in K^{\times}$ with $u\in\oo^{\times}$ and $t\in\ZZ$, we have
$$
I_{\pi}(\sigma,0)=
\begin{dcases}
1,&\text{ if }t\geq0,\\
\norm(\pp)^{\frac{t}{2}},&\text{ if }t<0,t\text{ is even},\\
\epsilon_{\pi}^3\legendre{u}{\pp}\norm(\pp)^{\frac{t}{2}},&\text{ if }t<0,t\text{ is odd}.
\end{dcases}
$$
\end{lemma}
\begin{proof}
If $t\geq0$, then it is clear that $I_{\pi}(\sigma,0)=1$. If $t<0$, we set
$$
k\coloneqq\left\lfloor\frac{1-t}{2}\right\rfloor\geq1,
$$
and then the integral splits as follows,
\begin{align*}
I_{\pi}(\sigma,0)=&\sum_{r\in\oo/\pp^k}\int_{r+\pp^k}e_{\pi}\left(\sigma x^2\right)\mathrm{d}x=\frac{1}{\norm(\pp)^k}\sum_{r\in\oo/\pp^k}\int_{\oo}e_{\pi}\left(\sigma(r+\pi^kx)^2\right)\mathrm{d}x\\
=&~\frac{1}{\norm(\pp)^k}\sum_{r\in\oo/\pp^k}e_{\pi}\left(\sigma r^2\right)\int_{\oo}e_{\pi}\left(2\pi^k\sigma rx\right)\mathrm{d}x=\frac{1}{\norm(\pp)^k}\sum_{r\in\oo/\pp^k}e_{\pi}\left(\sigma r^2\right)\chi_{\oo}(\pi^k\sigma r),
\end{align*}
where in the last equality we apply Lemma \ref{thm::padiccirclemethod}. If $t$ is even, then the summand is non-zero if and only if $r\in\pp^k$. Therefore, we have
$$
I_{\pi}(\sigma,0)=\frac{1}{\norm(\pp)^k}=\norm(\pp)^{\frac{t}{2}}.
$$
If $t$ is odd, then the summand is non-zero if and only if $r\in\pp^{k-1}$. So replacing $r$ in the summation by $\pi^{k-1}r$ and using Lemma \ref{thm::gausssumproperty}, we see that
$$
I_{\pi}(\sigma,0)=\norm(\pp)^{\frac{t}{2}-\frac{1}{2}}\sum_{r\in\oo/\pp}e_{\pi}\left(\frac{ur^2}{\pi}\right)=\epsilon_{\pi}^3\legendre{u}{\pp}\norm(\pp)^{\frac{t}{2}}.
$$
This finishes the proof.
\end{proof}

Next, we prove a formula in general for $I_{\pi}(\sigma,\tau)$.

\begin{lemma}
\label{thm::gaussintegral1pinhomogeneous}
For $\sigma\in K^{\times}$ and $\tau\in K$, we set $t\coloneqq\min(\ord_{\pp}(\sigma),\ord_{\pp}(\tau))$ and we have
$$
I_{\pi}(\sigma,\tau)=
\begin{dcases}
1,&\text{ if }t\geq0,\\
0,&\text{ if }t<0,\ord_{\pp}(\sigma)>\ord_{\pp}(\tau),\\
e_{\pi}(-\tau^2/4\sigma)I_{\pi}(\sigma,0),&\text{ if }t<0,\ord_{\pp}(\sigma)\leq\ord_{\pp}(\tau).
\end{dcases}
$$
\end{lemma}
\begin{proof}
If $t\geq0$, then it is clear that $I_{\pi}(\sigma,\tau)=1$. If $t<0$ and $\ord_{\pp}(\sigma)>\ord_{\pp}(\tau)$, then replacing $x$ by $x+\pi^{-t-1}y$ with $y\in\oo$, we see that
$$
I_{\pi}(\sigma,\tau)=e_{\pi}(\pi^{-t-1}\tau y)I_{\pi}(\sigma,\tau).
$$
Because $e_{\pi}(\pi^{-t-1}\tau y)\neq1$ for some $y\in\oo$, we see that $I_{\pi}(\sigma,\tau)=0$ in this case. Lastly, if $\ord_{\pp}(\sigma)\leq\ord_{\pp}(\tau)$, then we can apply the change of variable $x\to x-\tau/2\sigma$ to see that
$$
I_{\pi}(\sigma,\tau)=e_{\pi}(-\tau^2/4\sigma)I_{\pi}(\sigma,0),
$$
as desired.
\end{proof}

We also have to evaluate the following integral.

\begin{lemma}
\label{thm::twistlinearintegralp}
For $\sigma=u\pi^t\in K^{\times}$ with $u\in\oo^{\times}$ and $t\in\ZZ$, we have
$$
\int_{\oo^{\times}}\legendre{x}{\pp}e_{\pi}\left(\sigma x\right)\mathrm{d}x=\epsilon_{\pi}^3\legendre{u}{\pp}\norm(\pp)^{-\frac{1}{2}}\chi_{\pi^{-1}\oo^{\times}}(\sigma).
$$
\end{lemma}
\begin{proof}
By Lemma \ref{thm::padiccirclemethod}, we have
$$
\int_{\oo^{\times}}\legendre{x}{\pp}e_{\pi}\left(\sigma x\right)\mathrm{d}x=\sum_{r\in\kappa^{\times}}\legendre{r}{\pp}\int_{r+\pp}e_{\pi}\left(\sigma x\right)\mathrm{d}x=\frac{1}{\norm(\pp)}\sum_{r\in\kappa^{\times}}\legendre{r}{\pp}e_{\pi}\left(\sigma r\right)\chi_{\pp^{-1}}(\sigma).
$$
So the integral vanishes if $t\leq-2$. If $t\geq0$, then we have
$$
\int_{\oo^{\times}}\legendre{x}{\pp}e_{\pi}\left(\sigma x\right)\mathrm{d}x=\frac{1}{\norm(\pp)}\sum_{r\in\kappa^{\times}}\legendre{r}{\pp}=0.
$$
Finally, if $t=-1$, then we can apply Lemma \ref{thm::gausssumproperty} to see that
$$
\int_{\oo^{\times}}\legendre{x}{\pp}e_{\pi}\left(\sigma x\right)\mathrm{d}x=\frac{1}{\norm(\pp)}\sum_{r\in\kappa^{\times}}\legendre{r}{\pp}e_{\pi}\left(\frac{ur}{\pi}\right)=\frac{G_{\pi}(u)}{\norm(\pp)}=\epsilon_{\pi}^3\legendre{u}{\pp}\norm(\pp)^{-\frac{1}{2}},
$$
as desired.
\end{proof}

Now we are ready to state and prove a formula for local densities in non-dyadic cases.

\begin{theorem}
\label{thm::localdensityp}
Suppose that $n\in\oo$ and  $Q(\mathbf{x})=\sum_{1\leq i\leq r}(b_ix_i^2+c_ix_i)$ be an integral quadratic polynomial over $K$. For $1\leq i\leq r$, we set $t_i\coloneqq\min(\ord_{\pp}(b_i),\ord_{\pp}(c_i))$. Then we set
\begin{align*}
D_{\pp}\coloneqq\{1\leq i\leq r~|~\ord_{\pp}(b_i)>\ord_{\pp}(c_i)\},\\
N_{\pp}\coloneqq\{1\leq i\leq r~|~\ord_{\pp}(b_i)\leq\ord_{\pp}(c_i)\}.
\end{align*}
and set $\mathfrak{t}_d\coloneqq\min\{t_i~|~i\in D_{\pp}\}$. For any $t\in\ZZ$, we define
$$
\mathcal{L}_{\pp}(t)\coloneqq\{i\in N_{\pp}~|~t_i-t<0\text{ is odd}\},\quad\ell_{\pp}(t)\coloneqq|\mathcal{L}_{\pp}(t)|.
$$
We further define
$$
\mathfrak{n}\coloneqq n+\sum_{i\in N_{\pp}}\frac{c_i^2}{4b_i}.
$$
If $\mathfrak{n}\neq0$, we assume that $\mathfrak{n}=\mathfrak{u}_n\pi^{\mathfrak{t}_n}$ with $\mathfrak{u}_n\in\oo^{\times}$ and $\mathfrak{t}_n\in\ZZ$. Otherwise, we set $\mathfrak{t}_n\coloneqq\infty$. Then the local density of $Q$ at $\pp$ is given by
$$
\beta_{\pp}(n;Q)=1+\left(1-\frac{1}{\norm(\pp)}\right)\sum_{\substack{1\leq t\leq\min(\mathfrak{t}_d,\mathfrak{t}_n)\\\ell_{\pp}(t)\text{ even}}}\delta_{\pp}(t)\norm(\pp)^{\tau_{\pp}(t)}+\omega_{\pp}\delta_{\pp}(\mathfrak{t}_n+1)\norm(\pp)^{\tau_{\pp}(\mathfrak{t}_n+1)},
$$
where $\tau_{\pp}(t)$, $\delta_{\pp}(t)$, and $\omega_{\pp}$ are given by
$$
\tau_{\pp}(t)\coloneqq t+\sum_{\substack{i\in N_{\pp}\\t_i<t}}\frac{t_i-t}{2},\quad
\delta_{\pp}(t)\coloneqq\epsilon_{\pi}^{3\ell_{\pp}(t)}\prod_{i\in\mathcal{L}_{\pp}(t)}\legendre{u_i}{\pp}.
$$
with $u_i\coloneqq\pi^{-\ord_{\pp}(b_i)}b_i\in\oo^{\times}$ and
$$
\omega_{\pp}\coloneqq
\begin{dcases}
0,&\text{ if }\mathfrak{t}_n\geq\mathfrak{t}_d,\\
-\frac{1}{\norm(\pp)},&\text{ if }\mathfrak{t}_n<\mathfrak{t}_d,\ell_{\pp}(\mathfrak{t}_n+1)\text{ is even},\\
\epsilon_{\pi}\legendre{\mathfrak{u}_n}{\pp}\norm(\pp)^{-\frac{1}{2}},&\text{ if }\mathfrak{t}_n<\mathfrak{t}_d,\ell_{\pp}(\mathfrak{t}_n+1)\text{ is odd}.
\end{dcases}
$$
\end{theorem}
\begin{proof}
By Lemma \ref{thm::gaussintegral1p} and Lemma \ref{thm::gaussintegral1pinhomogeneous}, we can evaluate the local density as follows,
\begin{align*}
\beta_{\pp}(n;Q)=&~1+\sum_{t\geq1}\int_{\pi^{-t}\oo^{\times}}e_{\pi}\left(-\sigma n\right)\prod_{1\leq i\leq r}I_{\pi}(\sigma b_i,\sigma c_i)\mathrm{d}\sigma\\
=&~1+\sum_{1\leq t\leq\mathfrak{t}_d}\delta_{\pp}(t)\norm(\pp)^{\tau_{\pp}(t)}\int_{\oo^{\times}}\legendre{\sigma}{\pp}^{\ell_{\pp}(t)}e_{\pi}\bigg(-\frac{\sigma}{\pi^t}\bigg(n+\sum_{\substack{i\in N_{\pp}\\t_i<t}}\frac{c_i^2}{4b_i}\bigg)\bigg)\mathrm{d}\sigma\\
=&~1+\sum_{1\leq t\leq\mathfrak{t}_d}\delta_{\pp}(t)\norm(\pp)^{\tau_{\pp}(t)}\int_{\oo^{\times}}\legendre{\sigma}{\pp}^{\ell_{\pp}(t)}e_{\pi}\left(-\frac{\sigma\mathfrak{n}}{\pi^t}\right)\mathrm{d}\sigma,
\end{align*}
where in the last equality, we make use of the following fact:
$$
n+\sum_{\substack{i\in N_{\pp}\\t_i<t}}\frac{c_i^2}{4b_i}=\mathfrak{n}-\sum_{\substack{i\in N_{\pp}\\t_i\geq t}}\frac{c_i^2}{4b_i}\equiv\mathfrak{n}\pmod{\pp^t}.
$$
If $\ell_{\pp}(t)$ is even, applying Lemma \ref{thm::padiccirclemethod}, we have
$$
\int_{\oo^{\times}}e_{\pi}\left(\frac{-\sigma\mathfrak{n}}{\pi^t}\right)\mathrm{d}\sigma=\left(\int_{\oo}-\int_{\pp}\right)e_{\pi}\left(\frac{-\sigma \mathfrak{n}}{\pi^t}\right)\mathrm{d}\sigma=\chi_{\pp^t}(\mathfrak{n})-\frac{1}{\norm(\pp)}\chi_{\pp^{t-1}}(\mathfrak{n}).
$$
If $\ell_{\pp}(t)$ is odd, applying Lemma \ref{thm::twistlinearintegralp} instead, we have
$$
\int_{\oo^{\times}}\legendre{\sigma}{\pp}e_{\pi}\left(\frac{-\sigma\mathfrak{n}}{\pi^t}\right)\mathrm{d}\sigma=\epsilon_{\pi}\legendre{\mathfrak{u}_n}{\pp}\norm(\pp)^{-\frac{1}{2}}\chi_{\pi^{t-1}\oo^{\times}}(\mathfrak{n}).
$$
Combining these two cases, we conclude that
$$
\beta_{\pp}(n;Q)=1+\left(1-\frac{1}{\norm(\pp)}\right)\sum_{\substack{1\leq t\leq\min(\mathfrak{t}_d,\mathfrak{t}_n)\\\ell_{\pp}(t)\text{ even}}}\delta_{\pp}(t)\norm(\pp)^{\tau_{\pp}(t)}+\omega_{\pp}\delta_{\pp}(\mathfrak{t}_n+1)\norm(\pp)^{\tau_{\pp}(\mathfrak{t}_n+1)},
$$
as desired.
\end{proof}
\begin{remark}
If $Q$ is a quadratic form, then $c_i=0$ for $1\leq i\leq r$ and thus $\mathfrak{n}=n$. Moreover, we have $D_{\pp}=\varnothing$. So it is understood that $\mathfrak{t}_d\coloneqq\infty$. Then we can simplify $\omega_{\pp}$ as follows,
$$
\omega_{\pp}=
\begin{dcases}
-\frac{1}{\norm(\pp)},&\text{ if }\ell_{\pp}(\mathfrak{t}_n+1)\text{ is even},\\
\epsilon_{\pi}\legendre{\mathfrak{u}_n}{\pp}\norm(\pp)^{-\frac{1}{2}},&\text{ if }\ell_{\pp}(\mathfrak{t}_n+1)\text{ is odd}.
\end{dcases}
$$
\end{remark}

\section{Unramified dyadic cases}
\label{sec::unramifieddyadic}

In this section, we give an explicit formula for the local densities when $K$ is dyadic and unramified. By \cite[32:9 and 32:10]{o2013introduction}, the unique unramified extension over $\QQ_2$ of degree $f$ is isomorphic to $\QQ_2(\zeta)$, where $\zeta$ is a primitive $(2^f-1)$-th root of unity. For simplicity, we take $K=\QQ_2(\zeta)$. Then we denote by $U$ the subgroup of $K^{\times}$ generated by $\zeta$. For $u\in U$, we denote by $\sqrt{u}$ the unique square root of $u$. By \cite[32:8a]{o2013introduction}, the set $U_0\coloneqq U\cup\{0\}$ is a complete set of representatives of $\kappa$, and thus every element in $\oo$ can be written as a power series in $\pi$ with coefficients in $U_0$, where $\pi$ is a uniformizer of $K$. For simplicity, we take $\pi=2$. By \cite[(2.12) of Chapter 3]{neukirch2013algebraic}, the different ideal is $\oo$. Finally, we note that
\begin{equation}
\label{eqn::tracesquare}
\trace_{K/\QQ_2}(x^2)=\trace_{K/\QQ_2}(x),
\end{equation}
for $x\in U_0$.

\begin{lemma}
\label{thm::dyadicinhomogeneous}
For $\sigma\in\oo^{\times}$ and $\tau\in\oo$, we have
$$
\sum_{r\in\oo/\pp}e_{\pi}\left(\frac{\sigma r^2+\tau r}{2}\right)=\chi_{\pp}(\sigma-\tau^2)\norm(\pp).
$$
\end{lemma}
\begin{proof}
See \cite[Corollary 2]{gurak2010gauss}.
\end{proof}

\begin{lemma}
\label{thm::dyadicgausssum}
For any $\sigma\equiv a+2b\pmod{\pp^2}$ with $a\in U$ and $b\in U_0$, we have
$$
\sum_{r\in U_0}e_{\pi}\left(\frac{\sigma r}{4}\right)=\sum_{r\in U_0}e_{\pi}\left(\frac{\sigma r^2}{4}\right)=\sum_{r\in\oo/\pp}e_{\pi}\left(\frac{\sigma r^2}{4}\right)=e_{\pi}\left(\frac{a-2b}{8a}\right)(-1)^{f-1}\norm(\pp)^{\frac{1}{2}}.
$$
\end{lemma}
\begin{proof}
Using the fact that $r\mapsto r^2$ gives rise to a bijection from $U_0$ to $U_0$, we obtain the first equality. The second equality is obvious. The third equality follows from \cite[Lemma 7]{gurak2010gauss}.
\end{proof}

\begin{lemma}
\label{thm::teichmullersum}
For $a,b\in U$ such that $a+b\equiv u_0+2u_1\pmod{\pp}$ with $u_0,u_1\in U_0$, we have $u_0\equiv a+b+2\sqrt{ab}\pmod{\pp^2}$ and $u_1=\sqrt{ab}$. Furthermore, we have
$$
e_{\pi}\left(\frac{u_0}{4}\right)=e_{\pi}\left(\frac{a+b+2ab}{4}\right).
$$
For any $a,b,c\in U$ such that $a+b+c\equiv u_0\pmod{\pp}$ with $u_0\in U_0$, we have 
$$
e_{\pi}\left(\frac{u_0}{4}\right)=e_{\pi}\left(\frac{a+b+c+2ab+2ac+2bc}{4}\right).
$$
\end{lemma}
\begin{proof}
Since $(\sqrt{a}-\sqrt{b})^2\in U_0+\pp^2$, we have
$$
a+b\equiv u_0+2\sqrt{ab}\pmod{\pp^2}.
$$
It follows that $u_1=\sqrt{ab}$ and $u_0\equiv a+b+2\sqrt{ab}\pmod{\pp^2}$. Using (\ref{eqn::tracesquare}), we see that
$$
e_{\pi}\left(\frac{u_0}{4}\right)=e_{\pi}\left(\frac{a+b+2\sqrt{ab}}{4}\right)=e_{\pi}\left(\frac{a+b+2ab}{4}\right).
$$
Similarly, we have $(\sqrt{a}+\sqrt{b}+\sqrt{c})^2=u_0\pmod{\pp^2}$. Using (\ref{eqn::tracesquare}) again, we see that
$$
e_{\pi}\left(\frac{u_0}{4}\right)=e_{\pi}\left(\frac{a+b+c+2ab+2ac+2bc}{4}\right),
$$
as desired.
\end{proof}

We define a map $\eta\colon\oo^{\times}\to\{\pm1\}$ by
\begin{equation}
\label{eqn::defineeta}
\eta(u)\coloneqq e_{\pi}\left(\frac{c-b}{2a}\right),
\end{equation}
for $u\equiv a+2b+4c\pmod{\pp^3}$ with $a\in U$ and $b,c\in U_0$.

\begin{lemma}
\label{thm::character}
The map $\eta$ is a quadratic character.
\end{lemma}
\begin{proof}
We have to show that $\eta(u_1u_2)=\eta(u_1)\eta(u_2)$ for any units $u_1,u_2\in\oo^{\times}$. First, we assume that $u_1\equiv1+2b_1+4c_1\pmod{\pp^3}$ and $u_2\equiv1+2b_2+4c_2\pmod{\pp^3}$ with $b_1,c_1,b_2,c_2\in U_0$. Then we have
$$
u_1u_2\equiv 1+2(b_1+b_2)+4(c_1+c_2+b_1b_2)\pmod{\pp^3}.
$$
If $b_1=0$ or $b_2=0$, then it is easy to see that $\eta(u_1u_2)=\eta(u_1)\eta(u_2)$. If $b_1,b_2\neq0$, then applying Lemma \ref{thm::teichmullersum} to $b_1+b_2$, we see that
$$
\eta(u_1u_2)=e_{\pi}\left(\frac{c_1+c_2+b_1b_2+\sqrt{b_1b_2}-b_1-b_2}{2}\right)=e_{\pi}\left(\frac{c_1-b_1}{2}\right)e_{\pi}\left(\frac{c_2-b_2}{2}\right)=\eta(u_1)\eta(u_2).
$$
By definition, we have $\eta(u)=\eta(u/a)$ if $u\equiv a\pmod{\pp}$ with $a\in U$. Hence, in general, if $u_1\equiv a_1\pmod{\pp^3}$ and $u_2\equiv a_2\pmod{\pp^3}$ with $a_1,a_2\in U$, then we have
$$
\eta(u_1)\eta(u_2)=\eta(u_1/a_1)\eta(u_2/a_2)=\eta(u_1u_2/a_1a_2)=\eta(u_1u_2),
$$
as desired.
\end{proof}


Set $\mathbf{x}\coloneqq(x_1,\ldots,x_{r_1})$, $\mathbf{y}\coloneqq(y_{1,1},y_{1,2},\ldots,y_{r_2,1},y_{r_2,2})$, and $\mathbf{z}\coloneqq(z_{1,1},z_{1,2},\ldots,z_{r_3,1},z_{r_3,2})$ with $r_1+2r_2+2r_3=r$. Using Lemma \ref{thm::jordan} and ignoring the constant term, we may assume that the quadratic polynomial is of the form
\begin{align*}
Q(\mathbf{x},\mathbf{y},\mathbf{z})=\sum_{1\leq i\leq r_1}\left(b_ix_i^2+c_ix_i\right)&+\sum_{1\leq j\leq r_2}\left(b'_jy_{j,1}y_{j,2}+c'_{j,1}y_{j,1}+c'_{j,2}y_{j,2}\right)\\
&+\sum_{1\leq k\leq r_3}\left(b''_k(z_{k,1}^2+z_{k,1}z_{k,2}+\rho z_{k,2}^2)+c''_{k,1}z_{k,1}+c''_{k,2}z_{k,2}\right),
\end{align*}
with $b_i,c_i,b'_j,c'_{j,1},c'_{j,2},b''_k,c''_{j,1},c''_{j,2}\in\oo$ for $1\leq i\leq r_1$, $1\leq j\leq r_2$, and $1\leq k\leq r_3$, where $\rho\in\oo^{\times}$ is a fixed unit depending on the field $K$. Thus we can rewrite the local density as follows,
$$
\beta_{\pp}(n;Q)=\int_{K}e_{\pi}(-\sigma n)\prod_{1\leq i\leq r_1}I_{\pi}(\sigma b_i,\sigma c_i)\prod_{1\leq j\leq r_2}I'_{\pi}(\sigma b'_j,\sigma c'_{j,1},\sigma c'_{j,2})\prod_{1\leq k\leq r_3}I''_{\pi}(\sigma b''_k,\sigma c''_{k,1},\sigma c''_{k,2})\mathrm{d}\sigma,
$$
where $I_{\pi}$ is defined in (\ref{eqn::gaussintegral1}) and we further define
$$
I'_{\pi}(\sigma,\tau_1,\tau_2)\coloneqq\int_{\oo^2}e_{\pi}(\sigma y_1y_2+\tau_1y_1+\tau_2y_2)\mathrm{d}y_1\mathrm{d}y_2,
$$
and
$$
I''_{\pi}(\sigma,\tau_1,\tau_2)\coloneqq\int_{\oo^2}e_{\pi}(\sigma(z_1^2+z_1z_2+\rho z_2^2)+\tau_1z_1+\tau_2z_2)\mathrm{d}z_1\mathrm{d}z_2.
$$
Following the strategy in the non-dyadic cases, we proceed to evaluate the Gauss integrals $I_{\pi}$, $I'_{\pi}$, and $I''_{\pi}$. This will be achieved in a series of lemmas.

\begin{lemma}
\label{thm::gaussintegral12}
For $\sigma=u2^t\in K^{\times}$ with $u\in\oo^{\times}$ and $t\in\ZZ$, we have
$$
I_{\pi}(\sigma,0)=
\begin{dcases}
1,&\text{ if }t\geq0,\\
0,&\text{ if }t=-1,\\
(-1)^{(t+1)(f-1)}\eta(u)^{t+1}e_{\pi}(u^{2^f-1}/8)\norm(\pp)^{\frac{t+1}{2}},&\text{ if }t<-1.
\end{dcases}
$$
\end{lemma}
\begin{proof}
If $t\geq0$, then it is clear that $I_{\pi}(\sigma,0)=1$. If $t<0$, we set
$$
k\coloneqq\left\lfloor\frac{1-t}{2}\right\rfloor\geq1,
$$
and then the same argument in the proof of Lemma \ref{thm::gaussintegral1p} gives
$$
I_{\pi}(\sigma,0)=\frac{1}{\norm(\pp)^k}\sum_{r\in\oo/\pp^k}e_{\pi}\left(\sigma r^2\right)\chi_{\oo}(\pi^{k+1}\sigma r).
$$
Then the formula for $t=-1$ follows from Lemma \ref{thm::dyadicinhomogeneous}, and the formula for any even integer $t\leq-2$ follows from Lemma \ref{thm::dyadicgausssum} and Lemma \ref{thm::character}. Finally, if $t\leq-3$ and $t$ is odd, then we have
$$
I_{\pi}(\sigma,0)=\frac{1}{\norm(\pp)^k}\sum_{r\in\oo/\pp^2}e_{\pi}\left(\frac{ur^2}{8}\right)=\norm(\pp)^{\frac{t-1}{2}}\sum_{r\in\oo/\pp}e_{\pi}\left(\frac{ur^2}{8}\right)\sum_{\rho\in\oo/\pp}e_{\pi}\left(\frac{u\rho^2+ur\rho}{2}\right).
$$
By Lemma \ref{thm::dyadicinhomogeneous}, the inner sum vanishes if and only if $ur^2\equiv 1\pmod{\pp}$, which is equivalent to $r\equiv u^{2^{f-1}-1}\pmod{\pp}$. Thus, we have
$$
I_{\pi}(\sigma,0)=e_{\pi}(u^{2^f-1}/8)\norm(\pp)^{\frac{t+1}{2}},
$$
as desired.
\end{proof}

\begin{lemma}
\label{thm::gaussintegral12inhomogeneous}
For $\sigma\in K^{\times}$ and $\tau\in K$, we set $t\coloneqq\min(\ord_{\pp}(\sigma),\ord_{\pp}(\tau))$ and we have
$$
I_{\pi}(\sigma,\tau)=
\begin{dcases}
1,&\text{ if }t\geq0,\\
0,&\text{ if }t<0,\ord_{\pp}(\sigma)>\ord_{\pp}(\tau),\\
\chi_{\pp}(\pi\sigma-\pi^2\tau^2),&\text{ if }t=-1,\ord_{\pp}(\sigma)=\ord_{\pp}(\tau),\\
0,&\text{ if }t<-1,\ord_{\pp}(\sigma)=\ord_{\pp}(\tau),\\
e_{\pi}\left(-\tau^2/4\sigma\right)I_{\pi}(\sigma,0),&\text{ if }t<0,\ord_{\pp}(\sigma)<\ord_{\pp}(\tau).
\end{dcases}
$$
\end{lemma}
\begin{proof}
The proof is essentially the same as the proof of Lemma \ref{thm::gaussintegral1pinhomogeneous}, except the case $\ord_{\pp}(\sigma)=\ord_{\pp}(\tau)$. Assume that $\ord_{\pp}(\sigma)=\ord_{\pp}(\tau)=t<0$. We set
$$
k\coloneqq\left\lfloor\frac{1-t}{2}\right\rfloor\geq1.
$$
Then we have
$$
I_{\pi}(\sigma,\tau)=\frac{1}{\norm(\pp)^k}\sum_{r\in\oo/\pp^k}e_{\pi}(\sigma r^2+\tau r)\int_{\oo}e_{\pi}(\pi^kx(\tau+\pi\sigma r))\mathrm{d}x=\frac{\chi_{\oo}(\pi^k\tau)}{\norm(\pp)^k}\sum_{r\in\oo/\pp^k}e_{\pi}(\sigma r^2+\tau r).
$$
If $t<-1$, then we have $\pi^k\tau\not\in\oo$ and thus $I_{\pi}(\sigma,\tau)=0$. If $t=-1$, then the desired formula follows from Lemma \ref{thm::dyadicinhomogeneous}.
\end{proof}

\begin{lemma}
\label{thm::gaussintegral2inhomegeneous}
For $\sigma\in K^{\times}$ and $\tau_1,\tau_2\in K$, we set $t\coloneqq\min(\ord_{\pp}(\sigma),\ord_{\pp}(\tau_1),\ord_{\pp}(\tau_2))$ and we have
$$
I'_{\pi}(\sigma,\tau_1,\tau_2)=
\begin{dcases}
1,&\text{ if }t\geq0,\\
0,&\text{ if }t<0,\ord_{\pp}(\sigma)>\min(\ord_{\pp}(\tau_1),\ord_{\pp}(\tau_2)),\\
e_{\pi}\left(-\tau_1\tau_2/\sigma\right)\norm(\pp)^{t},&\text{ if }t<0,\ord_{\pp}(\sigma)\leq\min(\ord_{\pp}(\tau_1),\ord_{\pp}(\tau_2)).
\end{dcases}
$$
\end{lemma}
\begin{proof}
The first case is clear. For the second case, without loss of generality, we may assume that $\ord_{\pp}(\sigma)>\ord_{\pp}(\tau_1)$. Then by Lemma \ref{thm::padiccirclemethod} we have
\begin{equation}
\label{eqn::hyperbolicintermediate}
I'_{\pi}(\sigma,\tau_1,\tau_2)=\int_{\oo}e_{\pi}(\tau_2y_2)\int_{\oo}e_{\pi}((\sigma y_2+\tau_1)y_1)\mathrm{d}y_1\mathrm{d}y_2=\int_{\oo}e_{\pi}(\tau_2y_2)\chi_{\oo}(\sigma y_2+\tau_1)\mathrm{d}y_2.
\end{equation}
Since $\ord_{\pp}(\sigma)>\ord_{\pp}(\tau_1)$, we see that $\chi_{\oo}(\sigma y_2+\tau_1)=0$ for $y_2\in\oo$. This shows the second case. For the third case, replacing $y_2$ by $y_2-\tau_1/\sigma$ in (\ref{eqn::hyperbolicintermediate}) and using the assumption $\ord_{\pp}(\sigma)<\ord_{\pp}(\tau_2)$, we see that
$$
I'_{\pi}(\sigma,\tau_1,\tau_2)=e_{\pi}\left(-\frac{\tau_1\tau_2}{\sigma}\right)\int_{\oo}e_{\pi}(\tau_2y_2)\chi_{\oo}(\sigma y_2)\mathrm{d}y_2=e_{\pi}\left(-\frac{\tau_1\tau_2}{\sigma}\right)\norm(\pp)^t,
$$
as desired.
\end{proof}
\begin{remark}
The proof doesn't make use of the assumption that $K/\QQ_2$ is unramified. So the result remains valid for ramified dyadic extensions.
\end{remark}

\begin{lemma}
\label{thm::gaussintegral3inhomegeneous}
For $\sigma\in K^{\times}$ and $\tau_1,\tau_2\in K$, we set $t\coloneqq\min(\ord_{\pp}(\sigma),\ord_{\pp}(\tau_1),\ord_{\pp}(\tau_2))$ and we have
$$
I''_{\pi}(\sigma,\tau_1,\tau_2)=
\begin{dcases}
1,&\text{if }t\geq0,\\
0,&\text{if }t<0,\ord_{\pp}(\sigma)>\min(\ord_{\pp}(\tau_1),\ord_{\pp}(\tau_2)),\\
(-1)^{\trace_{K/\QQ_2}(\rho)t}e_{\rho}(\sigma,\tau_1,\tau_2)\norm(\pp)^t,&\text{if }t<0,\ord_{\pp}(\sigma)\leq\min(\ord_{\pp}(\tau_1),\ord_{\pp}(\tau_2)),
\end{dcases}
$$
where we set
$$
e_{\rho}(\sigma,\tau_1,\tau_2)\coloneqq e_{\pi}\left(\frac{\tau_1\tau_2-\rho\tau_1^2-\tau_2^2}{(4\rho-1)\sigma}\right).
$$
\end{lemma}
\begin{proof}
If $t\geq0$, then it is clear that $I''_{\pi}(\sigma,\tau_1,\tau_2)=1$. Suppose that $t<0$ in the following. First, we prove the lemma when $\tau_1=\tau_2=0$. Set
$$
k\coloneqq\left\lfloor\frac{1-t}{2}\right\rfloor\geq1.
$$
Then the same argument in the proof of Lemma \ref{thm::gaussintegral1p} gives
$$
\int_{\oo}e_{\pi}(\sigma(z_1z_2+\rho z_2^2))\mathrm{d}z_2=\frac{1}{\norm(\pp)^k}\sum_{r\in\oo/\pp^k}e_{\pi}(\sigma(z_1r+\rho r^2))\chi_{\pp^{-t-k}}(z_1+2\rho r).
$$
Thus, we see that
\begin{align*}
I''_{\pi}(\sigma,0,0)=&~\frac{1}{\norm(\pp)^k}\sum_{r\in\oo/\pp^k}\int_{\oo}e_{\pi}(\sigma(z_1^2+z_1r+\rho r^2))\chi_{\pp^{-t-k}}(z_1+2\rho r)\mathrm{d}z_1\\
=&~\frac{1}{\norm(\pp)^k}\sum_{r\in\oo/\pp^k}\int_{-2\rho r+\pp^{-t-k}}e_{\pi}(\sigma(z_1^2+z_1r+\rho r^2))\mathrm{d}z_1\\
=&~\norm(\pp)^t\sum_{r\in\oo/\pp^k}e_{\pi}((4\rho^2-\rho)\sigma r^2)\int_{\oo}e_{\pi}(\sigma(2^{-2t-2k}z_1^2+2^{-t-k}(1-4\rho)rz_1))\mathrm{d}z_1.
\end{align*}
If $t$ is even, then $t+2k=0$. By Lemma \ref{thm::padiccirclemethod} we see that
$$
I''_{\pi}(\sigma,0,0)=\norm(\pp)^t\sum_{r\in\oo/\pp^k}e_{\pi}((4\rho^2-\rho)\sigma r^2)\int_{\oo}e_{\pi}(2^{-t-k}(1-4\rho)\sigma rz_1)\mathrm{d}z_1=\norm(\pp)^t.
$$
If $t$ is odd, then $t+2k=1$. Write $\sigma=u2^t$ with $u\in\oo^{\times}$ and $t\in\ZZ$. Splitting the integral region $\oo=\cup_{s\in\oo/\pp}(s+\pp)$, using the change of the variable $z_1\to s+2z_1$, and then applying Lemma \ref{thm::padiccirclemethod} again, we have
$$
\int_{\oo}e_{\pi}(\sigma(2^{-2t-2k}z_1^2+2^{-t-k}(1-4\rho)rz_1))\mathrm{d}z_1=\frac{\chi_{\pp^{k-1}}(r)}{\norm(\pp)}\sum_{s\in\oo/\pp}e_{\pi}\left(\frac{us^2}{2}\right)e_{\pi}\left(\frac{(1-4\rho)urs}{2^k}\right).
$$
By Lemma \ref{thm::dyadicinhomogeneous}, we conclude that
\begin{align*}
I''_{\pi}(\sigma,0,0)=&~\norm(\pp)^{t-1}\sum_{r\in\pp^{k-1}/\pp^k}e_{\pi}((4\rho^2-\rho)\sigma r^2)\sum_{s\in\oo/\pp}e_{\pi}\left(\frac{us^2}{2}\right)e_{\pi}\left(\frac{(1-4\rho)urs}{2^k}\right)\\
=&~\norm(\pp)^{t-1}\sum_{r\in\oo/\pp}e_{\pi}\left(\frac{u\rho r^2}{2}\right)\sum_{s\in\oo/\pp}e_{\pi}\left(\frac{us^2+urs}{2}\right)=(-1)^{\trace_{K/\QQ_2}(\rho)}\norm(\pp)^t.
\end{align*}
Finally, we prove the formula in general. If $\ord_{\pp}(\sigma)>\min(\ord_{\pp}(\tau_1),\ord_{\pp}(\tau_2))$, without loss of generality, we suppose that $\ord_{\pp}(\tau_1)\geq\ord_{\pp}(\tau_2)$. By Lemma \ref{thm::gaussintegral12inhomogeneous}, we have
\begin{align*}
I''_{\pi}(\sigma,\tau_1,\tau_2)=&\int_{\oo}e_{\pi}(\sigma z_1^2+\tau_1z_1)\int_{\oo}e_{\pi}(\sigma\rho z_2^2+(\sigma z_1+\tau_2)z_2)\mathrm{d}z_2\mathrm{d}z_1\\
=&\int_{\oo}e_{\pi}(\sigma z_1^2+\tau_1z_1)I_{\pi}(\sigma\rho,\sigma z_1+\tau_2)\mathrm{d}z_1=0.
\end{align*}
If $\ord_{\pp}(\sigma)\leq\min(\ord_{\pp}(\tau_1),\ord_{\pp}(\tau_2))$, then applying the change of variables $z_1\to z_1+u_1$ and $z_2\to z_2+u_2$ with
$$
u_1\coloneqq\frac{\tau_2-2\rho\tau_1}{(4\rho-1)\sigma},\quad u_2\coloneqq\frac{\tau_1-2\tau_2}{(4\rho-1)\sigma},
$$
we see that
$$
I''_{\pi}(\sigma,\tau_1,\tau_2)=e_{\pi}\left(\frac{\tau_1\tau_2-\rho\tau_1^2-\tau_2^2}{(4\rho-1)\sigma}\right)I''_{\pi}(\sigma,0,0)=(-1)^{\trace_{K/\QQ_2}(\rho)t}e_{\rho}(\sigma,\tau_1,\tau_2)\norm(\pp)^t,
$$
as desired.
\end{proof}

Finally, we have to evaluate one more integral. For $a\in U$, $\alpha\in K$, $m\in\oo$, and $\ell\in\ZZ$, we set
$$
I_a^{\times}(\alpha,m,\ell)\coloneqq\int_{a+\pp}\eta(\sigma)^{\ell}e_{\pi}\left(\sigma\alpha+\sigma^{2^f-1}m/8\right)\mathrm{d}\sigma.
$$

\begin{lemma}
\label{thm::twistlinearintegral2I}
For $a\in U$, $\alpha\in K$, and $m\in(1+\pp)\cup\pp$, we have
$$
I_a^{\times}(\alpha,m,0)=\frac{\chi_{\pp^2}(8\alpha a+m)}{\norm(\pp)}e_{\pi}\left(\frac{8\alpha a+m}{8}\right).
$$
\end{lemma}
\begin{proof}
If $\sigma\equiv a+2b+4c\pmod{\pp^3}$ with $b,c\in U_0$, then we have
$$
\sigma^{2^f-1}\equiv1-\frac{2b}{a}+\frac{4b^2}{a^2}-\frac{4c}{a}\pmod{\pp^3}.
$$
It follows that
$$
e_{\pi}\left(\sigma^{2^f-1}m/8\right)=e_{\pi}\left(\frac{m}{8}+\frac{mb}{4a}+\frac{mc}{2a}\right).
$$
Therefore, we have
$$
I_a^{\times}(\alpha,m,0)=\frac{e_{\pi}(m/8)}{\norm(\pp)^3}\sum_{b,c\in U_0}e_{\pi}\left(\alpha(a+2b+4c)+\frac{mb}{4a}+\frac{mc}{2a}\right)\int_{\oo}e_{\pi}(8\alpha\sigma)\mathrm{d}\sigma.
$$
Using Lemma \ref{thm::padiccirclemethod} to evaluate the integral and then using Lemma \ref{thm::sumcirclemethod} to evaluate the sum over $c$, we have
$$
I_a^{\times}(\alpha,m,0)=\frac{\chi_{\pp}(8\alpha a+m)}{\norm(\pp)^2}e_{\pi}\left(\frac{8\alpha a+m}{8}\right)\sum_{b\in U_0}e_{\pi}\left(\frac{(8\alpha a+m)b}{4a}\right).
$$
If $8\alpha a+m\in\pp$, applying Lemma \ref{thm::sumcirclemethod} again gives
$$
I_a^{\times}(\alpha,m,0)=\frac{\chi_{\pp^2}(8\alpha a+m)}{\norm(\pp)}e_{\pi}\left(\frac{8\alpha a+m}{8}\right),
$$
as desired.
\end{proof}

\begin{lemma}
\label{thm::twistlinearintegral2J}
Suppose that $a\in U$, $\alpha\in K$, and $m\in(1+\pp)\cup\pp$. If $8\alpha\not\in\oo$, then we have $I_a^{\times}(\alpha,m,1)=0$. If $8\alpha\in\oo$, then we write $8\alpha\equiv\alpha_0+2\alpha_1\pmod{\pp^2}$ and $m\equiv m_0+2m_1\pmod{\pp^2}$ with $\alpha_0,\alpha_1,m_0,m_1\in U_0$. If $m\in1+\pp$, then we have
$$
I_a^{\times}(\alpha,m,1)=\frac{\chi_{\pp}(8\alpha)(-1)^{f-1}\eta(m)}{\norm(\pp)^{\frac{3}{2}}}e_{\pi}\left(\frac{(4\alpha+m\alpha_1)a}{4}\right).
$$
If $m\in\pp$, then we have
$$
I_a^{\times}(\alpha,m,1)=\frac{\chi_{1+\pp}(8\alpha a)(-1)^{f-1}\eta(8\alpha(1+m))}{\norm(\pp)^\frac{3}{2}}e_{\pi}\left(\frac{m_1\alpha_1}{2\alpha_0}\right).
$$
\end{lemma}
\begin{proof}
Using (\ref{eqn::defineeta}) and the same argument in the proof of Lemma \ref{thm::twistlinearintegral2I}, we see that if $8\alpha\not\in\oo$ then $I_a^{\times}(\alpha,m,1)=0$, and if $8\alpha\in\oo$, then we have
$$
I_a^{\times}(\alpha,m,1)=\frac{\chi_{\pp}(8\alpha a+m+1)}{\norm(\pp)^2}e_{\pi}\left(\frac{8\alpha a+m}{8}\right)\sum_{b\in U_0}e_{\pi}\left(\frac{(8\alpha a+m+2)b}{4a}\right).
$$
If $m\in1+\pp$, then we have $8\alpha a+m+1\in\pp$ if and only if $8\alpha\in\pp$. By Lemma \ref{thm::dyadicgausssum} and Lemma \ref{thm::teichmullersum}, we see that
$$
\sum_{b\in U_0}e_{\pi}\left(\frac{(8\alpha+m+2)b}{4a}\right)=\sum_{b\in U_0}e_{\pi}\left(\frac{(8\alpha a-m)b}{4}\right)=e_{\pi}\left(\frac{m\alpha_1a+m_1}{4}-\frac{1}{8}\right)(-1)^{f-1}\norm(\pp)^{\frac{1}{2}}.
$$
Using (\ref{eqn::defineeta}), we see that
$$
I_a^{\times}(\alpha,m,1)=\frac{\chi_{\pp}(8\alpha)(-1)^{f-1}\eta(m)}{\norm(\pp)^{\frac{3}{2}}}e_{\pi}\left(\frac{(4\alpha+m\alpha_1)a}{4}\right).
$$
Similarly, if $m\in\pp$, then we have $8\alpha a+m+1\in\pp$ if and only if $8\alpha a\in1+\pp$ and the sum over $b$ is equal to
$$
\sum_{b\in U_0}e_{\pi}\left(\frac{(8\alpha a+m+2)b}{4a}\right)=e_{\pi}\left(\frac{m_1+\alpha_1/\alpha_0+2m_1\alpha_1/\alpha_0}{4}\right)e_{\pi}\left(-\frac{1}{8}\right)(-1)^{f-1}\norm(\pp)^{\frac{1}{2}}.
$$
and by (\ref{eqn::defineeta}), we have
$$
I_a^{\times}(\alpha,m,1)=\frac{\chi_{1+\pp}(8\alpha a)(-1)^{f-1}\eta(8\alpha(1+m))}{\norm(\pp)^\frac{3}{2}}e_{\pi}\left(\frac{m_1\alpha_1}{2\alpha_0}\right),
$$
as desired.
\end{proof}

Now we state and prove a formula for local densities when $K/\QQ_2$ is unramified.

\begin{theorem}
\label{thm::localdensity2}
Suppose that $n\in\oo$ and $Q$ is an integral quadratic polynomial over $K$ of the form
\begin{align*}
Q(\mathbf{x},\mathbf{y},\mathbf{z})=\sum_{1\leq i\leq r_1}\left(b_ix_i^2+c_ix_i\right)&+\sum_{1\leq j\leq r_2}\left(b'_jy_{j,1}y_{j,2}+c'_{j,1}y_{j,1}+c'_{j,2}y_{j,2}\right)\\
&+\sum_{1\leq k\leq r_3}\left(b''_k(z_{k,1}^2+z_{k,1}z_{k,2}+\rho z_{k,2}^2)+c''_{k,1}z_{k,1}+c''_{k,2}z_{k,2}\right).
\end{align*}
For $1\leq i\leq r_1$, $1\leq j\leq r_2$, and $1\leq k\leq r_3$, we set $u_i\coloneqq\pi^{-\ord_{\pp}(b_i)}b_i$, $t_i\coloneqq\min(\ord_{\pp}(b_i),\ord_{\pp}(c_i))$, $t'_j\coloneqq\min(\ord_{\pp}(b'_j),\ord_{\pp}(c'_{j,1}),\ord_{\pp}(c'_{j,2}))$, and $t''_k\coloneqq\min(\ord_{\pp}(b''_k),\ord_{\pp}(c''_{k,1}),\ord_{\pp}(c''_{k,2}))$. Furthermore, we define the following sets:
\begin{align*}
D_{\pp}&\coloneqq\{1\leq i\leq r_1~|~\ord_{\pp}(b_i)>\ord_{\pp}(c_i)\},\\
E_{\pp}&\coloneqq\{1\leq i\leq r_1~|~\ord_{\pp}(b_i)=\ord_{\pp}(c_i)\},\\
N_{\pp}&\coloneqq\{1\leq i\leq r_1~|~\ord_{\pp}(b_i)<\ord_{\pp}(c_i)\},\\
D'_{\pp}\coloneqq\{1\leq &j\leq r_2~|~\ord_{\pp}(b'_j)>\min(\ord_{\pp}(c'_{j,1}),\ord_{\pp}(c'_{j,2}))\},\\
N'_{\pp}\coloneqq\{1\leq &j\leq r_2~|~\ord_{\pp}(b'_j)\leq\min(\ord_{\pp}(c'_{j,1}),\ord_{\pp}(c'_{j,2}))\},\\
D''_{\pp}\coloneqq\{1\leq &k\leq r_3~|~\ord_{\pp}(b''_k)>\min(\ord_{\pp}(c''_{k,1}),\ord_{\pp}(c''_{k,2}))\},\\
N''_{\pp}\coloneqq\{1\leq &k\leq r_3~|~\ord_{\pp}(b''_k)\leq\min(\ord_{\pp}(c''_{k,1}),\ord_{\pp}(c''_{k,2}))\},
\end{align*}
and set $\mathfrak{t}_d\coloneqq\min\{\{t_i~|~i\in D_{\pp}\}\cup\{t_i+1~|~i\in E_{\pp}\}\cup\{t'_j~|~j\in D'_{\pp}\}\cup\{t''_k~|~k\in D''_{\pp}\}\}$. For any $t\in\ZZ$, we define
$$
\mathcal{L}_{\pp}(t)\coloneqq\{i\in N_{\pp}~|~t_i-t+1<0\text{ is odd}\},\quad\ell_{\pp}(t)\coloneqq|\mathcal{L}_{\pp}(t)|.
$$
We further define
$$
\mathfrak{n}\coloneqq n+\sum_{i\in N_{\pp}}\frac{c_i^2}{4b_i}+\sum_{j\in N'_{\pp}}\frac{c'_{j,1}c'_{j,2}}{b'_j}+\sum_{k\in N''_{\pp}}\frac{\rho c_{k,1}''^2+c_{k,2}''^2-c''_{k,1}c''_{k,2}}{(4\rho-1)b''_k}.
$$
If $\mathfrak{n}\neq0$, we assume that $\mathfrak{n}=\mathfrak{u}_n\pi^{\mathfrak{t}_n}$ with $\mathfrak{u}_n\in\oo^{\times}$ and $\mathfrak{t}_n\in\ZZ$. Otherwise, we set $\mathfrak{t}_n\coloneqq\infty$. Then the local density of $Q$ at $\pp$ is given by
$$
\beta_{\pp}(n;Q)=1+\sum_{\substack{1\leq t\leq\min(\mathfrak{t}_d,\mathfrak{t}_n+3)\\t_i+1\neq t,i\in N_{\pp}}}\omega_{\pp}(t)\delta_{\pp}(t)\norm(\pp)^{\tau_{\pp}(t)-1},
$$
where $\tau_{\pp}(t)$ and $\delta_{\pp}(t)$ are given by
$$
\tau_{\pp}(t)\coloneqq t+\sum_{\substack{i\in N_{\pp}\\t_i+1<t}}\frac{t_i-t+1}{2}+\sum_{\substack{j\in N'_{\pp}\\t'_j<t}}(t'_j-t)+\sum_{\substack{k\in N''_{\pp}\\t''_k<t}}(t''_k-t),
$$
and
$$
\delta_{\pp}(t)\coloneqq\prod_{i\in\mathcal{L}_{\pp}(t)}(-1)^{f-1}\eta(u_i)\prod_{\substack{k\in N''_{\pp}\\t''_k-t<0\text{ odd}}}(-1)^{\trace_{K/\QQ_2}(\rho)}.
$$
For any integer $t\in\ZZ$, we set 
$$
U(t)\coloneqq\left\{u\in U~\bigg|~\frac{c_i}{\pi^{t-1}}\equiv\frac{ub_i^2}{\pi^{2t-2}}\pmod{\pp}\text{ for any }i\in E_{\pp}\text{ with }t_i+1=t\right\}\subseteq U.
$$
In case that $|U(t)|=1$, we denote by $u(t)$ the unique element in $U(t)$. Furthermore, we define
$$
m(t)\coloneqq\sum_{\substack{i\in N_{\pp}\\t_i+1<t}}u_i^{2^f-1}\in(1+\pp)\cup\pp.
$$
For any $1\leq t\leq\min(\mathfrak{t}_d,\mathfrak{t}_n+3)$, we write $-8\mathfrak{n}/\pi^t\equiv\mathfrak{n}_0(t)+2\mathfrak{n}_1(t)\pmod{\pp^2}$ and $m(t)\equiv m_0(t)+2m_1(t)\pmod{p^2}$ with $\mathfrak{n}_0(t),\mathfrak{n}_1(t)$, $m_0(t),m_1(t)\in U_0$. Then we define $\omega_{\pp}(t)$ as follows:
\begin{enumerate}
\item If $|U(t)|=0$, then $\omega_{\pp}(t)=0$.
\item If $\ell_{\pp}(t)$ is even and $|U(t)|=1$, then we have
$$
\omega_{\pp}(t)=\chi_{\pp^2}\left(m(t)-\frac{8u(t)\mathfrak{n}}{\pi^t}\right)e_{\pi}\left(\frac{m(t)}{8}-\frac{u(t)\mathfrak{n}}{\pi^t}\right).
$$
\item If $\ell_{\pp}(t)$ is even, $|U(t)|=|U|$, and $m(t)\in1+\pp$, then we have
$$
\omega_{\pp}(t)=\chi_{\pp^2}\left(\mathfrak{n}_0(t)m(t)-\frac{8\mathfrak{n}}{\pi^t}\right)e_{\pi}\left(\frac{m(t)}{8}-\frac{\mathfrak{n}}{\mathfrak{n}_0(t)\pi^t}\right).
$$
\item If $\ell_{\pp}(t)$ is even, $|U(t)|=|U|$, $m(t)\in\pp$, and $t\leq\mathfrak{t}_n+1$, then we have
$$
\omega_{\pp}(t)=\chi_{\pp^2}(m(t))e_{\pi}\left(\frac{m(t)}{8}\right)\left(\chi_{\pp^t}(\mathfrak{n})\norm(\pp)-1\right).
$$
\item If $\ell_{\pp}(t)$ is even, $|U(t)|=|U|$, $m(t)\in\pp$, and $t\geq\mathfrak{t}_n+2$, then we have
$$
\omega_{\pp}(t)=\chi_{\pi^{t-3}\oo^{\times}}\left(\frac{\mathfrak{n}}{m(t)}\right)e_{\pi}\left(\frac{m(t)}{8}-\frac{m_1(t)\mathfrak{n}}{\mathfrak{n}_1(t)\pi^t}\right).
$$
\item If $\ell_{\pp}(t)$ is odd, $|U(t)|=1$, and $m(t)\in1+\pp$, then we have
$$
\omega_{\pp}(t)=\frac{\chi_{\pp^{t-2}}(\mathfrak{n})(-1)^{f-1}\eta(m(t))}{\norm(\pp)^{\frac{1}{2}}}e_{\pi}\left(\frac{\mathfrak{n}_1(t)u(t)m(t)}{4}-\frac{u(t)\mathfrak{n}}{\pi^t}\right).
$$
\item If $\ell_{\pp}(t)$ is odd, $|U(t)|=|U|$, and $m(t)\in1+\pp$ then we have
$$
\omega_{\pp}(t)=\frac{\chi_{\pp^{t-2}}(\mathfrak{n})(-1)^{f-1}\eta(m(t))}{\norm(\pp)^{\frac{1}{2}}}\left(\chi_{\pp^2}\left(\mathfrak{n}_1(t)m(t)-\frac{4\mathfrak{n}}{\pi^t}\right)\norm(\pp)-1\right).
$$
\item If $\ell_{\pp}(t)$ is odd and $m(t)\in\pp$, then we have $\omega_{\pp}(t)=0$ if $\mathfrak{n}_0(t)^{-1}\not\in U(t)$. Otherwise we have
$$
\omega_{\pp}(t)=\frac{\chi_{\pi^{t-3}\oo^{\times}}(\mathfrak{n})(-1)^{f-1}}{\norm(\pp)^{\frac{1}{2}}}\eta\left(\frac{8\mathfrak{n}(1+m(t))}{\pi^t}\right)e_{\pi}\left(\frac{m_1(t)\mathfrak{n}_1(t)}{2\mathfrak{n}_0(t)}\right).
$$
\end{enumerate}
\end{theorem}
\begin{proof}
By Lemma \ref{thm::gaussintegral12inhomogeneous}, Lemma \ref{thm::gaussintegral2inhomegeneous}, and Lemma \ref{thm::gaussintegral3inhomegeneous}, we have
$$
\beta_{\pp}(n;Q)=1+\sum_{\substack{1\leq t\leq\mathfrak{t}_d\\t_i+1\neq t,i\in N_{\pp}}}\delta_{\pp}(t)\norm(\pp)^{\tau_{\pp}(t)}\sum_{a\in U(t)}I_a^{\times}\left(-\frac{\mathfrak{n}}{\pi^t},m(t),\ell_{\pp}(t)\right).
$$
Since $I_a^{\times}\left(\alpha,m,\ell\right)=0$ if $8\alpha\not\in\oo$, we have
$$
\beta_{\pp}(n;Q)=1+\sum_{\substack{1\leq t\leq\min(\mathfrak{t}_d,\mathfrak{t}_n+3)\\t_i+1\neq t,i\in N_{\pp}}}\delta_{\pp}(t)\norm(\pp)^{\tau_{\pp}(t)}\sum_{a\in U(t)}I_a^{\times}\left(-\frac{\mathfrak{n}}{\pi^t},m(t),\ell_{\pp}(t)\right).
$$
Using Lemma \ref{thm::dyadicgausssum}, Lemma \ref{thm::twistlinearintegral2I}, and Lemma \ref{thm::twistlinearintegral2J}, it follows from a straightforward computation that
$$
\sum_{a\in U(t)}I_a^{\times}\left(-\frac{\mathfrak{n}}{\pi^t},m(t),\ell_{\pp}(t)\right)=\frac{\omega_{\pp}(t)}{\norm(\pp)},
$$
which finishes the proof.
\end{proof}
\begin{remark}
When $Q$ is a quadratic form, we have $E_{\pp}=N_{\pp}=N'_{\pp}=N''_{\pp}=\emptyset$. Thus it is understood that $\mathfrak{t}_d\coloneqq\infty$. In this case, we have $|U(t)|=|U|$ for any $t\in\ZZ$ and $\mathfrak{n}=n$. Therefore, one has a much simpler formula for the local density of a quadratic form. 
\end{remark}

\bibliographystyle{plain}
\bibliography{reference}

\end{document}